\documentclass[11pt,reqno]{amsart}
\usepackage{geometry}
\usepackage[latin1]{inputenc}
\usepackage[italian,english]{babel}
\usepackage{amsmath, amsfonts, amsthm,xcolor}
\usepackage{paralist}
\numberwithin{equation}{section}
\usepackage{mathtools, mathabx}
\usepackage{tikz}

\newtheorem{thm}{Theorem}[section]
\newtheorem{lem}[thm]{Lemma}

\newtheorem{prop}[thm]{Proposition}

\newtheorem{defn}[thm]{Definition}
\newtheorem{con}[thm]{Conjecture}
\theoremstyle{definition}
\newtheorem{rem}[thm]{Remark}
\theoremstyle{remark}

\newcommand{\ds}{\displaystyle}

\newcommand{\R}{\mathbb{R}}

\newcommand{\de}{\partial}


\date{}

\title{A reverse quantitative isoperimetric type  inequality for the Dirichlet Laplacian}
\author[  G. Paoli]{
 Gloria Paoli}
\address{Dipartimento di Matematica e Applicazioni ``R. Caccioppoli'', Universit\`a degli studi di Napoli Federico II \\ Via Cintia, Complesso Universitario Monte S. Angelo, 80126 Napoli, Italy.}
\email{gloria.paoli@unina.it}

\begin{document}
\maketitle

\markright{A REVERSE QUANTITATIVE ISOPERIMETRIC INEQUALITY}
\markleft{A REVERSE QUANTITATIVE ISOPERIMETRIC INEQUALITY}

\begin{abstract} A stability result in terms of the perimeter is  obtained for the first Dirichlet eigenvalue of the Laplacian operator.
In particular,  we prove that, once we  fix the dimension $n\geq2$, there exists a constant $c>0$, depending only on  $n$,  such that, for every $\Omega\subset\mathbb{R}^n$ open, bounded and convex set with volume equal to the volume of a ball $B$ with radius $1$,  it holds 
\begin{equation*}
	\lambda_1(\Omega)-\lambda_1(B)\geq c\left(P(\Omega)-P(B)  \right)^{2},
\end{equation*}
where by $\lambda_1(\cdot)$ we denote  the first Dirichlet eigenvalue of a set and  by $P(\cdot)$ its perimeter.  The hearth of the present paper is a sharp estimate of the Fraenkel asymmetry in terms of the perimeter.
\end{abstract}

\section{Introduction}

Let $\Omega\subset\mathbb{R}^n$, with $n\geq2$, be an open set with finite Lebesgue measure. The first Dirichlet-Laplacian eigenvalue associated to $\Omega$, that we denote by $\lambda_1(\Omega)$,  is the least positive $\lambda$ such that  
\begin{equation*}\label{dir_intro}
\begin{cases}
-\Delta u=\lambda u & \mbox{in }\Omega\\
u=0 & \mbox{on }\partial \Omega
\end{cases}
\end{equation*}
admits  non-trivial solutions in $H^1_0(\Omega)$.
The classical result of Faber and Krahn  states that, among measurable  domains with fixed measure,  $\lambda_1(\cdot)$ is minimized by a ball (see \cite{Faber, Krahn, PS}).  In other words, the following  scaling invariant inequality holds:
\begin{equation}\label{Faber-Krahn_intro}
    \lambda_1(\Omega) |\Omega|^{2/n}\geq  \lambda_1(B_R) |B_R|^{2/n},
\end{equation}
where by $|\Omega|$ we denote the volume of a measurable set  $\Omega$, i.e. its Lebesgue measure,  and by $B_R$ a geniric ball in $\mathbb{R}^n$ of radius $R$. Equality in \eqref{Faber-Krahn_intro} holds if and only if $\Omega$ is equivalent to  a ball. 
 Moreover, the quantity $\lambda_1 (\Omega)$ can be variationaly  characterized by
 \begin{equation*}\label{min}
\lambda_{1}( \Omega )=\min_{\substack{v\in H^1_0(\Omega) \\ v\nequiv 0}}\dfrac{\ds\int_{\Omega} |Dv|^2\;dx}{\ds\int_{\Omega}v^2\;dx}
\end{equation*}
and has the following scaling property, for every $t>0$,
\begin{equation*}
\lambda_1(t^2 \Omega)=t^{-2}\lambda_1(\Omega).
\end{equation*}
We refer to \cite{H} and \cite{K} as survay books on the argument. 

In this paper we consider  the following class of admissible sets
\begin{equation*}
\mathcal{C}_n:=\{ \Omega\subseteq\mathbb{R}^n\;|\; \Omega \text{ convex},\; |\Omega|=|B|  \},
\end{equation*}
where $B$ is a  ball of $\mathbb{R}^n$ with radius  equal to  $1$. From now on,  we  denote by $P(\Omega)$ the perimeter of $\Omega$ in the De Giorgi sense and by $\omega_n$ the volume of the ball $B$. Our starting point is  the following conjecture in the planar case, that is stated in  \cite{FL}.

\begin{con} \cite{FL} \label{ilias}
Let $\Omega$ be an open, bounded and convex subset of $\mathbb{R}^2$ with $|\Omega|=\omega_n$,
then
\begin{equation}\label{ilias_eq}
\lambda_1(\Omega)-\lambda_1(B)\geq \beta \left(P(\Omega)- P(B ) \right)^{3/2},
\end{equation}
where   $\beta:=\frac{4\cdot3^{3/2}\;\zeta(3)}{\pi^{9/2}}$ and $\zeta(n)=\sum_{k=1}^{\infty} k^{-n}$ is the Riemann zeta function.
\end{con}
This  conjecture is supported by  numerical and analytical results, in particular we refer to  \cite{GS, GS1, M}. 
\begin{prop}\cite{GS}
Let $\mathcal{P}^*_k$ be the regular polygon with $k$ edges and area equal to $\pi$. Then, as $k$ goes to $+\infty$, 
\begin{equation*}
\lambda_1(\mathcal{P}^*_k)-\lambda_1(B)\sim\beta\left(P(\mathcal{P}_k^*)-P(B)  \right)^{3/2},
\end{equation*}
where $\beta$ is the constant defined in Conjecture \ref{ilias}. 
\end{prop}
The method used by Grinfeld and Strang in \cite{GS, GS1} to prove this fact comes from  differential geometry and relays on the calculus of moving surfaces. 
This result  was also proved by Molinari in \cite{M}, using the Schwartz-Christoffel mappings, that are useful tools to  express the Dirichlet  eigenvalue of a polygon as a series expansion, relating each expansion term to a summation over Bessel functions.

By the way, we recall that the fundamental tone of the Dirichlet-Laplacian on polygons has been widely investigated and, nevertheless, many questions are still unsolved. The famous  Poly\'a-Szeg{\"o} conjecture \cite{PS},for example,  states   that among all the $k$-gons of given area the regular one achieves the least possible $\lambda_1$ and has been settled only for $k=3$ and $k=4$, that are the only cases for which it is possible to use the Steiner symmetrization (see \cite{H, PS}).

Conjecture \ref{ilias} is also supported by numerical observations, linked  to the plot of the Blaschke-Santal\'o diagram for the triplet $(P(\cdot),\,\lambda_1(\cdot),\,|\cdot|\;)$, that  is the sets of points
$$ \{\left(P(\Omega),\lambda_1(\Omega)\right)\;|\; \; \Omega\in\mathcal{C}_2  \}.$$ 
For these numerical results we refer to \cite{AF, FL}. In particular, in \cite{FL}, the authors, by generating  random polygons woth sides between $3$ and $30$, find out that regular polygons lay on the lower part of the diagram, while the shapes that lay on the upper part of the boundary of the diagram are smooth, expecting  that stadii are a good approximation of the upper part.


In this work we are not able to prove Conjecture \ref{ilias} in the plane. Instead, we  prove a less strong result, but in  every dimension. The main  Theorem is the following.

\begin{thm}\label{main_thm}
	Let $n\geq 2$. 	There exists a  constant   $c>0$,  depending only on $n$,  such that, for every $\Omega\in\mathcal{C}_n$, it holds
	\begin{equation}\label{main}
	\lambda_1(\Omega)-\lambda_1(B)\geq c \left(P(\Omega)-P(B) \right)^2,
	\end{equation}
	being $B$  a ball of $\mathbb{R}^n$ with radius equal to $1$.
\end{thm}

In order to obtain this, we  prove an intermediate result, that is the core of this work and that has its own interest:
	there exists a constant   $C>0$, depending only on the dimension $n$,  such that, for every $\Omega\in\mathcal{C}_n$, it holds
	\begin{equation}\label{first_ineq_intro}
	\mathcal{A}_F(\Omega)\geq C \left(  P(\Omega)-P(B) \right),
	\end{equation}
	where  $\mathcal{A}_F(\Omega)$ is the Fraenkel asymmetry of $\Omega$ (a $L^1-$distance between sets), defined as 
	\begin{align}\label{fraenkel_def}
	\mathcal{A}_F(\Omega)=\min_{x\in\mathbb{R}^n}\bigg \{ \dfrac{|\Omega\Delta B_R(x)|}{|\Omega|}   \;:\:\text{$B_R(x)$ is a ball s.t. $|B_R(x)|=|\Omega|$}\bigg  \},
	\end{align}
	denoting by $\Delta $  the symmetric  difference between two sets and by $B_R(x)$ the ball of radius $R$ centered at the point $x$.
	Inequality \eqref{first_ineq_intro} is true provided  that the isoperimetric deficit $P(\Omega)-P(B)$ is small enough. 

	We prove \eqref{main} combing \eqref{first_ineq_intro} with the sharp  stability result for the Faber-Krahn inequality, proved in  \cite{BDPV}, that states  that 
	there exists a constant $\bar{C}>0$ such that for every open set $\Omega$ with $|\Omega|=\omega_n$, the following inequality holds:
	\begin{equation}\label{quantitative_intro}
	\lambda_1(\Omega)-\lambda_1 (B)\geq \bar{C}\mathcal{A}_F(\Omega)^2.
	\end{equation}
	
	Unfortunately, as we will show  providing  a class of counter-examples, inequality \eqref{first_ineq_intro} is not true when the difference of perimeters has exponent $3/4$, that is the target power to obtain \eqref{ilias_eq}, if we combine it with \eqref{quantitative_intro}. The exponent $1$ of the difference of the perimeters in \eqref{first_ineq_intro} is, indeed, an optimal exponent, since it is reached asymptotically by regular polygons with number of sides that goes to infinity.  This is the reason for which we cannot use this strategy  to prove  Conjecture \ref{ilias}. So, this work represents a step forward to the resolution of this still open problem, generalising and obtaining some results  also in dimension greater than $2$.

Moreover, we can say that inequality \eqref{first_ineq_intro} is a sort of reverse quantitative inequality.  We recall indeed that  the sharp quantitative isoperimetric inequality, proved in \cite{FMP} in any dimension, asserts that there exists a constant $K>0$ such that, for every Borel set $\Omega\subset\mathbb{R}^n$ with $|\Omega|=\omega_n$,
\begin{equation}\label{fusco-maggi-pratelli_intro}
	 \left( P(\Omega)-P(B)\right)\geq K \mathcal{A}_F^2(\Omega).
\end{equation}
On the contrary to \eqref{fusco-maggi-pratelli_intro}, in our inequality \eqref{main} the terms of the difference of the perimeter is used as an asymmetry functional and, so, it is situated  in the left hand side of the inequality, obtaining, in this way, a quantitative isoperimetric inequality for $\lambda_1(\cdot)$ in terms of the difference of the perimeter.

As a consequence of the fact that the Fraenkel asymmetry is a bounded quantity,  more precicelly $\mathcal{A}_F(\Omega)\in[0,2)$ and is zero if and only if $\Omega $ is a ball,  inequality \eqref{first_ineq_intro} cannot be true when $\Omega$ is, for example,  a long and flat domain and this is the reason for which, in the non local case,  the proof of inequality \eqref{main} will be  treated separatelly.

$ $ \\
The paper is organized as follows. In Section $2$ we recall some preliminary definitions and results about convex sets and some  spectral inequalities, proving also a lower bound of the Hausdorff distance in terms of the perimeter.  Section $3$ is devoted to the proof of  the sharp geometric inequality \eqref{first_ineq_intro}, that is the core of the present paper. Finally, in Section $4$ we prove the main Theorem \ref{main_thm}.

\section{Notations and preliminary results}  
\subsection{Notations and basic definitions}

We will use the following notations: $\cdot$ is the Euclidean scalar product in $\mathbb{R}^n$, 
$B$ is a generic ball of radius equal to $1$,  $B(O)$ is the ball with radius equal to  $1$ centered at the fixed origin $O=(0,\dots,0)$ and   $\mathbb{S}^{n-1}$ is  the unit sphere in $\mathbb{R}^n$. Moreover, we denote by $|\cdot|$ the Lebesgue measure of a set and  by   $\mathcal{H}^k$, for $k\in [0,n)$, the $k-$dimensional Hausdorff measure in $\mathbb{R}^n$. 
The perimeter of $\Omega$ in $\mathbb{R}^n$ will be denoted by $P(\Omega)$ and, if $P(\Omega)<\infty$, we say that $\Omega$ is a set of finite perimeter. In our case $\Omega$ is a convex and bounded set and this ensures us that it is a set of finite perimeter and that $P(\Omega)=\mathcal{H}^{n-1}(\partial\Omega)$.

For simplicity, we introduce the following notation:
\begin{equation*}
\Delta \lambda_1(\Omega):=\lambda_1(\Omega)-\lambda_1(B), \qquad \Delta P(\Omega):=P(\Omega)-P(B).
\end{equation*}

 We give  the definition  of tangential gradient, that can be found in \cite{B}.
\begin{defn} \label{tangential_euclidean_gradient}
	Let $\Omega$ be an open, bounded subset of $\mathbb{R}^n$ with Lipschitz boundary and let $u:\mathbb{R}^n\rightarrow\mathbb{R}$ be a Lipschitz function. We can define the tangential gradient of $u$ for almost every $x\in \de\Omega$ as
	$$ \nabla_{\tau}u(x)=\nabla u(x)-\langle \nabla u(x),\nu_{\de\Omega}(x)\rangle\nu_{\de\Omega}(x),$$
	where $\nu_{\de\Omega}$ is the outer unit normal to $\de\Omega$,	whenever $\nabla u$  exists at $x$.
\end{defn}
 From now on we will consider open, bounded  and convex subsets of $\mathbb{R}^n$, with $n\geq 2$.
and,  for the content  of the following,   we refer mainly  to \cite{S}. We start by  recalling  the definition of Hausdorff distance  between two open, bounded and convex sets  $\Omega,K\subset\mathbb{R}^n$, that is:
\begin{equation}
\label{disth}
d_{\mathcal{H}}(\Omega,K)=\inf \left\{  \varepsilon>0  \; :\; \Omega\subset K+B_{\varepsilon},\; K\subset \Omega+B_{\varepsilon}\right\}, 
\end{equation}
where by $+$  we denote the Minkowski sum between sets and $B_\varepsilon$ is a geniric ball of radius $\varepsilon$.
Note that  we have $d_\mathcal{H}(\Omega,K)=d_\mathcal{H} (\partial \Omega, \partial K)$ and that the following rescaling property holds
\[
d_\mathcal{H}(t\Omega,tK)=t\,d_\mathcal{H}(\Omega,K), \quad t>0.
\]
We recall the definition of support function of a convex set.
\begin{defn}\label{support}
	Let $\Omega$ be an open, bounded and convex set of $\mathbb{R}^n$. The support function associated to $\Omega$ is defined, for every $y\in \mathbb{R}^n$, as follows:
	\begin{equation}\label{support_eq}
	h_\Omega(y)=\max_{x\in \Omega}\left(x\cdot y\right).
	\end{equation}
\end{defn}
If there is no possibility of confusion, we use the notation $h$ instead of $h_\Omega$. It is easy to see that the support function associated to a ball of radius $R$ centered at the origin is constantly equal to $R$. In the following remark the relation between the Haussdorf distance and the support function is pointed out.
\begin{rem}
	Let $K,\Omega$ be two open,  bounded and convex subsets  of $\mathbb{R}^n$. It holds:
	\begin{equation*}
	d_{\mathcal{H}}(\Omega,K)=|| h_\Omega-h_K||_{L^{\infty}(\mathbb{S}^{n-1})}.
	\end{equation*}
\end{rem}

We  define now  another   kind  asymmetry, usually used  for convex sets, different from the one recalled  in \eqref{fraenkel_def}. 
\begin{defn} Let $\Omega$ be an open, bounded and convex subset of $\mathbb{R}^n$. We define the Hausdorff asymmetry as
	\begin{equation}\label{hausdorf_def}
	\mathcal{A_\mathcal{H}}(\Omega)= \min_{x\in\mathbb{R}^n}\bigg \{\frac{d_{\mathcal{H}}(\Omega,B_R(x)) }{R}  \;:\:\text{$B_R(x)$ is a ball centered at $x$  s.t. $|B_R(x)|=|\Omega|$}\bigg  \}.
	\end{equation} 
\end{defn}


Finally, if  $\Omega\in\mathcal{C}_n$  is such that  the origin $O\in\Omega$, then its boundary can be parametrized as :
\begin{equation}\label{boundary}
\de \Omega=\{   y\in\mathbb{R}^n\;|\;y=r(x); \; x\in\mathbb{S}^{n-1} \}
\end{equation}
and we have the following results  (see for the exact computations \cite{F}):
\begin{equation}\label{asi_convex}
\mathcal{A}_F(\Omega)=\frac{1}{n} \ds\int_{\mathbb{S}^{n-1}} |r^n(x)-1|\;d\mathcal{H}^{n-1}(x)
\end{equation}
and 
\begin{equation}\label{peri_convex}
P(\Omega)=\int_{\mathbb{S}^{n-1}}\sqrt{r^{2(n-1)}(x)+r^{2(n-2)}(x)  |\nabla_{\tau}r|^2 }\:d\mathcal{H}^{n-1}.
\end{equation}
In particular, if $\Omega\in\mathcal{C}_2$, we can use the classical polar coordinates representation, that this
\begin{equation}\label{gauge}
\Omega=\bigg\{  (r,\theta) \in [0,\infty)\times [0,2\pi)\;|\;r<\frac{1}{u(\theta)}\bigg\},
\end{equation}
where $u$ is a positive and $2\pi-$periodic function, often called the gauge function of $\Omega$, for a reference see e.g \cite{S,LNP}. It is well known that $\Omega$ is convex if and only if $u''+u\geq 0$.
Moreover,  we  can write the volume, the perimeter and the Fraenkel asymmetry of $\Omega$ in terms of $u$:
\begin{align}\label{per_vol}
& |\Omega|=\frac{1}{2}\ds\int_{0}^{2\pi}\dfrac{1}{u^2(\theta)}\;d\theta;\\
& P(\Omega)=\ds\int_{0}^{2\pi}\dfrac{\sqrt{u^2+u'^2}}{u^2}d\theta
\end{align}
and, setting $r(\theta):=1/u(\theta)$, we have
\begin{equation}\label{still_long}
\mathcal{A}_F(\Omega)=\ds\int_{0}^{2\pi}\big |\ds\int_{r(\theta)}^{1} r\;dr|\;d\theta=\ds\int_{0}^{2\pi}\dfrac{|1-r^2(\theta)|}{2}d\theta=\frac{1}{2}\ds\int_{0}^{2\pi}\dfrac{|1-u^2(\theta)|}{u^2(\theta)}\;d\theta.
\end{equation}

$ $ \\

\subsection{Definition and some properties on quermassintegrals}

We give in this Section  the definition of quermassintegrals, referring  to \cite{S}. Let $ \Omega$ be  an open, bounded and convex subset  of $\mathbb{R}^n$.
We define the outer  parallel body of $\Omega$ at distance $\rho$ as the  Minkowski sum
$$ \Omega+\rho B=\{ x+\rho y\in\mathbb{R}^n\;|\; x\in \Omega,\;y\in B \}.$$
The Steiner formula asserts that
\begin{equation}\label{general_steiner}
|\Omega+\rho B|=\sum_{i=0}^{n}\binom{n}{i} W_i(\Omega)\rho^i. 
\end{equation}
and
\begin{equation}\label{general_steiner}
P(\Omega+\rho B)=n\sum_{i=0}^{n-1}\binom{n}{i} W_{i+1}(\Omega)\rho^i,
\end{equation}
where the coefficients $W_i(\Omega)$ are known as quermassintegrals. In particular, we have:
\begin{equation}
W_0(\Omega)=|\Omega|;\qquad n W_1(\Omega)=P(\Omega);\qquad W_n(\Omega)=\omega_n
\end{equation}
and
\begin{equation}\label{mean_1}
W_{n-1}(\Omega)=\dfrac{\omega_n}{2}\: \omega(\Omega),
\end{equation}
where $\omega(\Omega)$ is  called mean width of the convex body $\Omega$ and it is defined as 
\begin{equation}\label{mean}
\omega(\Omega)=\frac{1}{n\omega_n}\int_{\mathbb{S}^{n-1}}\left(h_\Omega(x)+h_\Omega(-x)\right)\;d\mathcal{H}^{n-1}(x),
\end{equation}
representing the mean value over all possible directions of the distance between parallel supporting hyperplanes to $\Omega$. 
Finally, we recall   the Aleksandrov-Fenchel inequalities:
\begin{equation}\label{aleksandrov-fenchel}
\left(\dfrac{W_j(\Omega)}{\omega_n}\right)^{\frac{1}{n-j}}\geq \left(\dfrac{W_i(\Omega)}{\omega_n}\right)^{\frac{1}{n-i}},
\end{equation}
for $0\leq i<j<n$, with equality if and only if $\Omega$ is a ball. 

In the following  Proposition, using the tools introduced in this Section,  we prove a lower bound of the Haussdorf distance in terms of the perimeter, that will be usefull for  the proof of the main Theorem.

\begin{prop}\label{per_Haus} 
	Let $\Omega$ be an open, bounded and convex subset of $\mathbb{R}^n$, $n\geq 2$, with $|\Omega|=\omega_n$ and containing the origin $O$. Then 
 	\begin{align}\label{z}
	d_{\mathcal{H}}(\Omega,B) \geq K   P(\Omega)^{\frac{2-n}{n-1}}   \Delta P(\Omega)
	\end{align}
	where $K$ is a  positive constants depending only on the dimension $n$ and $B$ is the  ball centered at the origin with radius equal to $1$.
\end{prop}

\begin{proof}
 We have
\begin{equation*}
d_{\mathcal{H}}(\Omega,B)=||h_\Omega-h_B ||_\infty\geq ||h_\Omega||_\infty-1,
\end{equation*}
where the quantity in the right hand side is positive, since $||h_\Omega||_{\infty}=\max\{|x|: x\in \Omega \}$ and having fixed the volume $|\Omega|=\omega_n$ . Consequently,
\begin{equation}\label{a}
\frac{1}{n\omega_n}\int_{\mathbb{S}^{n-1}}h_\Omega(x)\;d\mathcal{H}^{n-1}(x)\leq  d_{\mathcal{H}}(\Omega,B)+1.
\end{equation}
On the other hand, we have also that 
\begin{align}\label{d}
d_{\mathcal{H}}(\Omega,B) & \geq \frac{1}{n\omega_n}\int_{\mathbb{S}^{n-1}} |h_{\Omega}(x)-h_B(x)|\:d\mathcal{H}^{n-1}(x)\\&\geq \omega(\Omega)-\frac{1}{n\omega_n} \int_{\mathbb{S}^{n-1}}h_\Omega(x)\;d\mathcal{H}^{n-1}(x)- 1
\end{align}
So, by\eqref{a} and \eqref{mean_1}, we obtain 
\begin{equation}\label{dd}
d_{\mathcal{H}}(\Omega,B) \geq c_n \dfrac{W_{n-1}(\Omega)-W_{n-1}(B)}{2}
\end{equation}
and, using the  Alexandrov-Fenchel inequality  \eqref{aleksandrov-fenchel} for  $j=n-1$ and $i=1$, we have
\begin{equation}\label{af}
W_{n-1}(\Omega)\geq \dfrac{n^{\frac{n}{n-1}}}{\omega_n^{1/(n-1)}}   P(\Omega)^{1/(n-1)},
\end{equation}
where the equality if and only if $\Omega$ is a ball. 
So, combing  \eqref{dd} with \eqref{af}, we obtain
\begin{align}\label{z}
d_{\mathcal{H}}(\Omega,B) \geq c_n \left( P(\Omega)^{1/(n-1)}-   P(B)^{1/(n-1)} \right)\geq K_n P(\Omega)^{\frac{2-n}{n-1}} \Delta P(\Omega),
\end{align}
where $c_n, K_n$ are positive constants depending only on the dimension $n$.

\end{proof}

\begin{rem}\label{Haus_per2 }

	If we are in dimension $n=2$, since 
	\begin{equation*}\label{per_sup}
	P(\Omega)=\int_0^{2\pi} h(\theta)d\theta,
	\end{equation*}
	where $h$ is the support function associated to $\Omega$ as defined in \eqref{support_eq}, we have
	\begin{equation}\label{den_do}
	d_{\mathcal{H}}(\Omega, B)=|| h_\Omega-1||_{\infty},\geq \frac{1}{2\pi}\int_0^{2\pi}|h_\Omega(\theta)-1|\geq\frac{1}{2\pi}\Delta P(\Omega).
	\end{equation}

\end{rem}

\subsection{Spectral inequalities: background material }

We recall now some spectral results that we will use in the sequel. The first of them is the sharp  quantitative version of the Faber-Krahn inequality,  proved in \cite{BDPV}, where the sharpness is verified using a suitable family of ellipsoids.  
\begin{thm}[Sharp quantitative Faber-Krahn inequality]\label{quant_krahn_d}
	Let $n\geq 2$.	There exists a constant $\bar{C}>0$, depending only on $n$, such that, for every open set $\Omega$ with $|\Omega|=\omega_n$, it holds
	\begin{equation}\label{quantitative_v}
	\lambda_1(\Omega)-\lambda_1(B)\geq \bar{C} \mathcal{A}_F(\Omega)^2
	\end{equation} 
	and the exponent $2$ is sharp.
\end{thm}

A quantitative type result for the first Dirichlet  eigenvalue was  proved before   by Melas  in \cite{Me} for  the case of convex sets in $\mathbb{R}^n$, with a non-sharp exponent (for a reference see also \cite{brasco_survey} and \cite{FMP1}).

\begin{thm}[Quantitative Faber-Krahn inequality]\label{quant_krahn_d}
	Let $n\geq 2$.	There exists a constant $\bar{c}>0$, depending only on $n$, such that, for every  $\Omega\subseteq \mathbb{R}^n$ open bounded and convex, with $|\Omega|=\omega_n$, it holds
	\begin{equation}\label{quantitative_v}
	\lambda_1(\Omega)-\lambda_1(B)\geq \bar{c}\; d_\mathcal{M}(\Omega)^{2n},
	\end{equation} 
where 
\begin{equation*}\label{melly}
d_{\mathcal{M}}(\Omega):=\min \Big \{    \max     \Big\{    \dfrac{|\Omega\setminus B_1|}{\Omega} ,\; \dfrac{|B_2\setminus\Omega|}{B_2}\Big \} \;:\: B_1\subset\Omega\subset B_2 \; {\rm balls}          \Big\}.
\end{equation*}
\end{thm}

From \cite{brasco_survey} and \cite{EFT},  we have that 
\begin{equation}\label{Melas_Haus}
\lambda_1(\Omega)-\lambda_1(B)\geq M \mathcal{A}_H(\Omega)^{2n^2},
\end{equation}
with $M$ positive constant dependending only by the dimension

Finally, we recall the following lower bound of the first Dirichlet eigenvalue in term of the perimeter, proved in \cite{brasco}. 
\begin{thm}
	Let $n\geq 2$.	There exists a constant $K>0$, depending only on $n$, such that, for every open set $\Omega$ with $|\Omega|=\omega_n$, it holds
	\begin{equation}\label{brasco_thm}
	\lambda_1(\Omega)\geq K P(\Omega)^2
	\end{equation}
	and the exponent $2$ is sharp.
\end{thm}

\section{A sharp estimate of the Fraenkel asymmetry in terms of the perimeter}
This Section is dedicated to the proof of the following  Proposition, that is the hearth of the present paper. 
\begin{prop}\label{first_step} 
	Let $n\geq 2$.	There exist two  constants   $ \delta>0$ and $C>0$ depending only on $n$,  such that, for every $\Omega\in\mathcal{C}_n$ with  $\Delta P(\Omega)< \delta$,  it holds
	\begin{equation}\label{first_ineq}
	\mathcal{A}_F(\Omega)\geq C \left(  P(\Omega)-P(B) \right).
	\end{equation}
\end{prop}

 First of all we prove inequality  \eqref{first_ineq} in the planar case and after we generalize this result in all dimensions. 
 
 Moreover, we have that  inequality \eqref{first_step} is sharp, in the sense that the exponent of the Fraenkel asymmetry in \eqref{first_ineq} cannot be improved, as it is pointed out in the following Remark.
 
 \begin{rem}\label{sharp}
 We show  that 
 	$$ \mathcal{A}_{\mathcal{F}}(\mathcal{P}^*_k)\sim\left(P(\mathcal{P}^*_k)-P(B) \right),$$
 	when $\mathcal{P}_k^*$ is a regular $k$-gon of area $1$. Indeed, the  following relations hold:
 	\begin{equation*}
 	\mathcal{A}_F(\mathcal{P}^*_k)\simeq\epsilon^2, \qquad\qquad P(\mathcal{P}^*_k)-P(B)\simeq\epsilon^2,
 	\end{equation*}
 	where $\epsilon=\pi/k$ . Being $B$ the ball of area  $1$, if we set  $R$  the radius of $B$, we have that $R=1/\sqrt{\pi}$.  
 	Let us denote by  $a$  the apothem of $\mathcal{P}_k^*$, i.e. the segment  from the center of the polygon that is perpendicular to one of its sides.
 	Setting and Taylor expanding, we obtain that 
 	\begin{equation}\label{apo}
 	a=\dfrac{\left(1-\sin^2(\pi/k)  \right)^{1/4}}{\sqrt{k \sin(\pi/k)}}\simeq\dfrac{\left(1-\epsilon^2\right)^{1/4}}{\sqrt{\pi}}\simeq \dfrac{1-\epsilon^2/4}{\sqrt{\pi}}.
 	\end{equation}
 	Now the area $A_{\gamma}$ of the circular segment of angle $2\gamma$, that is the region of $B$  which is cut off from the rest by one edge of the polygon, 
 	$$A_\gamma=\frac{1}{2\pi}(\gamma-\sin(\gamma)) \simeq \frac{\sqrt{2}}{\pi}\epsilon^3,$$ 
 	since $\gamma=\arccos(a\sqrt{\pi})\simeq\sqrt{2}\epsilon.$
 	Thus,
 	\begin{equation}
 	\mathcal{A}_F(\mathcal{P}^*_k)\simeq2\sqrt{2\pi}\epsilon^2.
 	\end{equation}
 	Using \eqref{apo}  we obtain
 	\begin{equation*}
 	P(\mathcal{P}^*_k)-P(B)=\dfrac{2}{a}-2\sqrt{\pi}\simeq \frac{\sqrt{\pi}}{2}\epsilon^2.
 	\end{equation*}
 	On the other hand, sets that are smooth without edges do not create problems. If we consider for example  the family of ellipses 
 	$$E_\epsilon=\Big\{  (x,y)\:|\:x^2+\dfrac{y^2}{(1+\epsilon)^2}=1 \Big\} $$
 	we have that  (see the computations in  \cite{BDR})
 	$$ \mathcal{A}_{\mathcal{F}}(E_\epsilon)\sim\left(P(E_\epsilon)-P(B) \right)^{1/2},$$
 	being
 	\begin{equation*}
 	P(E_\epsilon)-P(B)\simeq\epsilon^2,\qquad\qquad   \mathcal{A}_{\mathcal{F}}(E_\epsilon)=O(\epsilon). 
 	\end{equation*}
 \end{rem}

\subsection{Proof of Proposition \ref{first_step} in the planar case}\label{planny}

Let us consider the case $n=2$ and let $\Omega\in\mathcal{C}_2$.  We denote by $B(x_\infty)$ the unit ball centered at the point $x_{\infty}$ that realizes the minimum in \eqref{hausdorf_def} and by $B(x_1)$ the unit ball centered at the point $x_1$ that realizes the minimum in \eqref{fraenkel_def}.  Without loss of generality, since convex sets which agree up to rigid motions (translations and rotations) can be systematically identified,  we can assume that $O\in\Omega$ and that $x_1$ coincides with the origin $O$.   We denote by $r$ the polar fuction centered at the origin as defined in  \eqref{boundary} and by $r_\infty$
  the polar function associated to $\Omega$ centered at $x_\infty$ (we can assume that $x_1,x_\infty\in \Omega$).

Let us fix $\varepsilon>0$. Using the result cointained in \cite{gelli} (Lemma $2.11$), we have that  there exists $\delta_1>0$, such that, if $\Delta P(\Omega)<\delta_1$, then  
\begin{equation}\label{fusco}
d_{\mathcal{H}}(B(x_1), B(x_\infty))< \varepsilon.
\end{equation}
Moreover,  there exists $\delta_0>0$ such that $\delta_0<\delta_1^{1/2}$ and such that, if $\mathcal{A}_H(\Omega)<\delta_0$, then  $||r_\infty-1||_{\infty}<\varepsilon$.  Now, let us assume that $|| 1-r||_{\infty}\geq|| 1-r_\infty||_\infty$ (otherwise we have that $|| 1-r||_{\infty}<\varepsilon$). In this case, from \eqref{fusco}, we have
\begin{equation}\label{small}
	|| 1-r||_{\infty}\leq || 1-r_\infty||_\infty+d_{\mathcal{H}}(B(x_1), B(x_\infty))< 2\varepsilon.
\end{equation}
Choosing $\delta=\delta_0^2$, we have that, if $\Delta P(\Omega)<\delta$, then $||1-r||_\infty<2 \varepsilon$  (Theorem $1.1.$ in \cite{gelli}).


Now,  recalling the definiton of gauge function given in \eqref{gauge},   the formulas \eqref{per_vol} \eqref{still_long} and using \eqref{small}, we have that \eqref{first_step} follows from 
\begin{equation}\label{long}
\int_{0}^{2\pi}|1-u|(1+u)\;d\theta=\int_{0}^{2\pi} |1-u^2(\theta)|\;d\theta\geq C\int_{0}^{2\pi}\left(   u'^2(\theta)+  u^2(\theta)-u^4(\theta)\right)\;d\theta, 
\end{equation} 
after a rationalization of the denominator and after taking care of adjusting the constant in \eqref{still_long}, that for convenience we still call $C$. 
Moreover, \eqref{long} follows from 
\begin{equation}\label{ineqq}
\ds\int_{0}^{2\pi} |1-u(\theta)|\;d\theta\geq C\ds\int_{0}^{2\pi}\left(u'(\theta)\right)^2d\theta.
\end{equation}

Our strategy to prove \eqref{first_step} is to prove  \eqref{ineq} for  a suitable  class of polygons  $\mathcal{P}$  near the unit  ball centerd at the origin and, then, we can conclude that  \eqref{ineqq} is true for every $\Omega$ such that  $\Delta P(\Omega)\leq \delta$ by a standard density argument in the Hausdorff metric.   More precisely, let us consider $\Omega$   such that  $\Delta P(\Omega)\leq \delta$. Then, for every $n\in \mathbb{N}$ with $n>\tilde{n}$ big enough, there exists a polygon  $\mathcal{P}_n$ such that 
  $d_\mathcal{H}(\Omega, \mathcal{P}_n)<1/n$  (see Theorem $1.8.16$ in \cite{S}) and such that 
\eqref{ineqq} holds for the gauge function $u_n$ associated to $\mathcal{P}_n$ (that is the claim of the following Lemma).   Then, we can conclude the proof of Proposition \ref{first_step}, since $\sup |u_n-u|\to 0$ (see  Lemma $2.1$ in \cite{Chen}).
So, it remains to prove only the following

\begin{lem}
There exists $\bar{n}\in\mathbb{N}$ and $C>0$ such that, for every $\mathcal{P}\in \mathcal{C}_2$  polygon  such that  $d_{\mathcal{H}}(\mathcal{P}, B)<1/n$ for  $n\geq \bar{n}$,  we have 
\begin{equation}\label{ineq} 
\ds\int_{0}^{2\pi} |1-u(\theta)|\;d\theta\geq C\ds\int_{0}^{2\pi}\left(u'(\theta)\right)^2d\theta,
	\end{equation}
where	$u$ is the gauge function associated  to $\mathcal{P}$.
\end{lem}

\begin{proof}
 In the following we are assuming that the segment  $\overline{AB}\subset\de \mathcal{P}$ is one side of the polygon $\mathcal{P}$ and we call $\theta_0$ the angle $A\hat{O}B$. 
 We denote by $P$ the point of intersection between the segment $\overline{AB}$ and the ray that forms with the segment $\overline{AO}$ an angle of amplitude $\theta$. We analyse now all the possible cases.\\
\textit{Case 1: External Tangent.}  We set $\alpha:=O\hat{A}B$ and we  assume that $\overline{OB}$ has length equal to $1$, that $\alpha+\theta_0=\pi/2$ and that the segment $\overline{OA}$ lies on the $y-$axis. Moreover, we set $h:=|\overline{AO}|-1$.

\begin{figure}[htpb]
	\centering
	\includegraphics[scale=0.90]{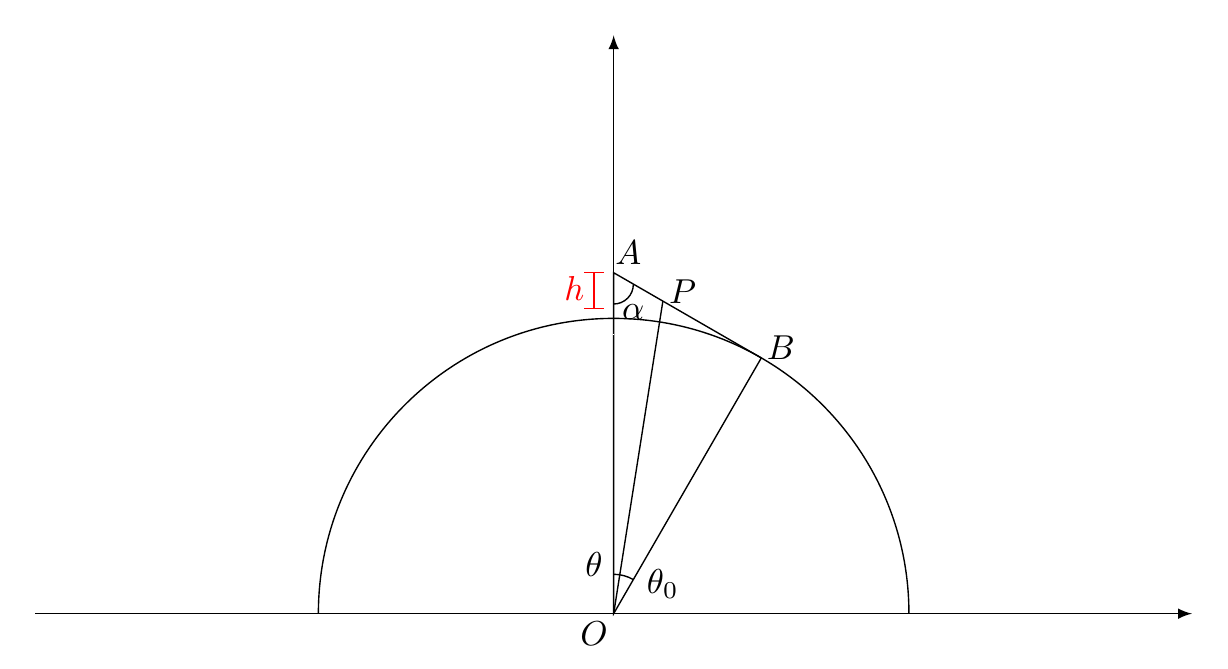}
	\caption{External tangent case}
	\label{fig1}
\end{figure}

We  have that 
\begin{equation}\label{1_sine}
h+1=\dfrac{1}{\cos(\theta_0)}.
\end{equation}
Let $\theta\in(0,\theta_0)$.
 Having $u(\theta)=1/r(\theta)$,  using the sine theorem to the triangle $OAP$, 
we obtain
\begin{equation}\label{sine_thm1}
	u(\theta)=\cos(\theta_0-\theta).
\end{equation}
and, consequently, 
\begin{equation}\label{derivative1}
	u'(\theta)=- \sin(\theta_0-\theta).
\end{equation}
Using \eqref{derivative1} and Taylor expanding up to the first order with respect to $\theta_0$, we have that

\begin{align}\label{second_term}
&\int_{0}^{\theta_0} \left( u'(\theta)\right)^2d\theta=\int_{0}^{\theta_0} \sin^2(\theta_0-\theta)\:d\theta= \left[ \frac{\theta_0}{2} -\dfrac{\sin(2\theta_0)}{4} \right]\approx  \,\frac{ 2\theta_0^3}{3}
\end{align}
and
\begin{align}\label{first_term}
	\ds\int_{0}^{\theta_0}|1-u(\theta)| \;d\theta=	\ds\int_{0}^{\theta_0}\left(1-u(\theta)\right)d\theta=\ds\int_{0}^{\theta_0}\left( 1-\cos(\theta_0-\theta)  \right)d\theta=\\
	=\left(\theta_0-\sin(\theta_0)\right)\approx \dfrac{\theta_0^3}{6}.
\end{align}
So,  there exists a non negative constant $C$ such that
\begin{equation*}
	\ds\int_{0}^{\theta_0}|1-u(\theta)| \;d\theta\geq C\int_{0}^{\theta_0} \left( u'(\theta)\right)^2d\theta.
\end{equation*}\\
\textit{Case 2: External Secant.} We are now considering the case when $\alpha+\theta_0<\pi/2$, the segment  $\overline{OB}$ has length $1$ and the side of the polygon $\overline{AB}$ is external to the unitary ball  and touches it in the point $B$.

\begin{figure}[htpb]
	\centering
	\includegraphics[scale=0.90]{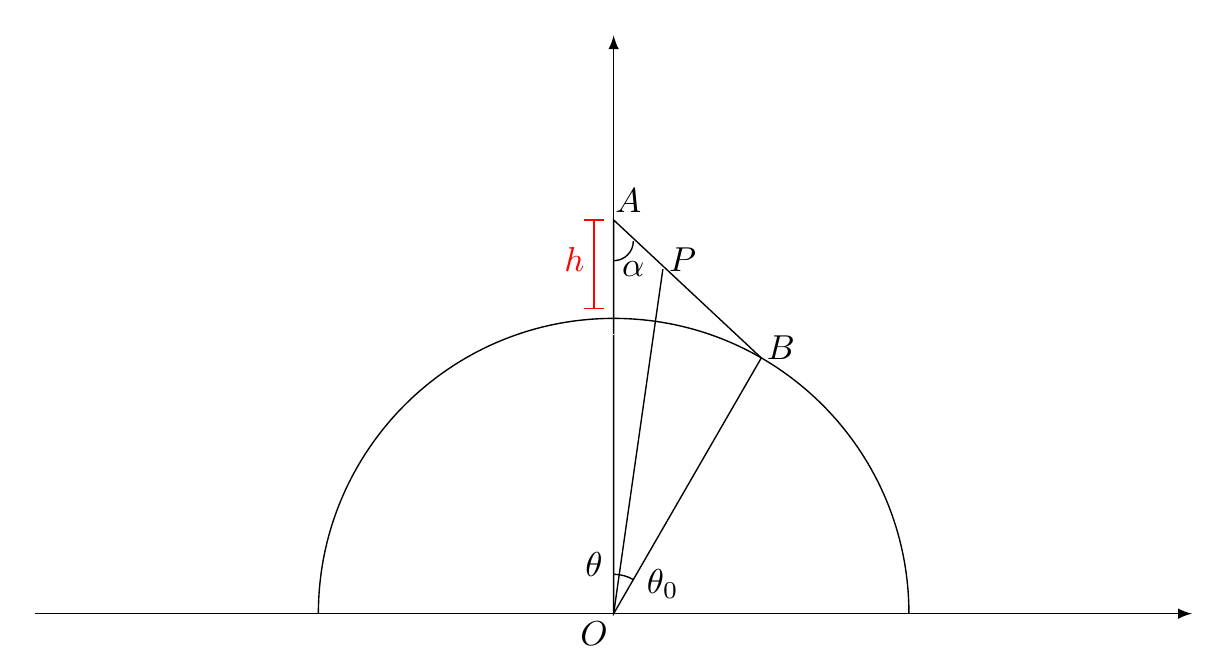}
	\caption{External Secant Case}
	\label{fig1}
\end{figure}
	We set $z=\pi/2-\alpha$ and in this case we have that 
	\begin{equation}\label{secondd}
h+1=\dfrac{\sin(\alpha+\theta_0)}{\sin(\alpha)}
	\end{equation}
	and,  using the sine theorem to the triangle $OAP$, 
	\begin{equation}\label{sine_thm}
	u(\theta)=\dfrac{\sin(\alpha+\theta)}{\sin(\alpha)(h+1)}
	\end{equation}
and so
	\begin{equation}\label{derivative}
	u'(\theta)=\dfrac{\cos(\alpha+\theta)}{\sin(\alpha)(h+1)}.
	\end{equation}
Denoting  by $I=1/\left( \sin(\alpha)(h+1)  \right)$ and Taylor expanding up to the first order with respect to $\theta_0$ and to  $z$, we have that

	\begin{align}\label{second_term}
	&\int_{0}^{\theta_0} \left( u'(\theta)\right)^2d\theta= I^2\left(\frac{\theta_0}{2} + \dfrac{\sin(2z-2\theta_0)}{4} -\dfrac{\sin (2z)}{4} \right)\simeq   \frac{I^2}{3} \left[3z^2\theta_0-3z\theta_0^2+\theta_0^3\right].
	\end{align}
	and, using \eqref{secondd}, 
	\begin{align*}\label{first_term}
	& \ds\int_{0}^{\theta_0}|1-u(\theta)| \;d\theta=	\ds\int_{0}^{\theta_0}\left(1-u(\theta)\right)d\theta=I\ds\int_{0}^{\theta_0}\left( 1-\dfrac{\cos(\theta_0-\theta)}{(h+1)\sin(\alpha)}  \right)d\theta=\\
	= & I\left[ \sin(\alpha)(h+1)\theta_0+\cos(\alpha+\theta_0)-\cos(\alpha)\right]= \\&=I \left[\theta_0\cos(\theta_0-z)+\sin(z)\left(\cos(\theta_0)-1\right)-\cos(z)\sin(\theta_0) \right] \simeq \dfrac{3I}{2} \left[   3z\theta_0^2-2\theta_0^3\right] . 
	\end{align*}
Since we can choose  $\theta_0 $ such that $z>(2/3)\theta_0$,  from the last two relations, we can conclude that there exists a non negative  constant $C$ such that
	\begin{equation*}
	\ds\int_{0}^{\theta_0}|1-u(\theta)| \;d\theta\geq C\int_{0}^{\theta_0} \left( u'(\theta)\right)^2d\theta.
	\end{equation*}\\

	\textit{Case 3: External not Intersecting.} We are  considering the case when $\alpha+\theta_0\leq \pi/2$, the segment  $\overline{OB}$ has length strictly greater than $1$ and the side of the polygon  $\overline{AB}$ is external to the unitary ball.
\begin{figure}[htpb]
	\centering
	\includegraphics[scale=0.90]{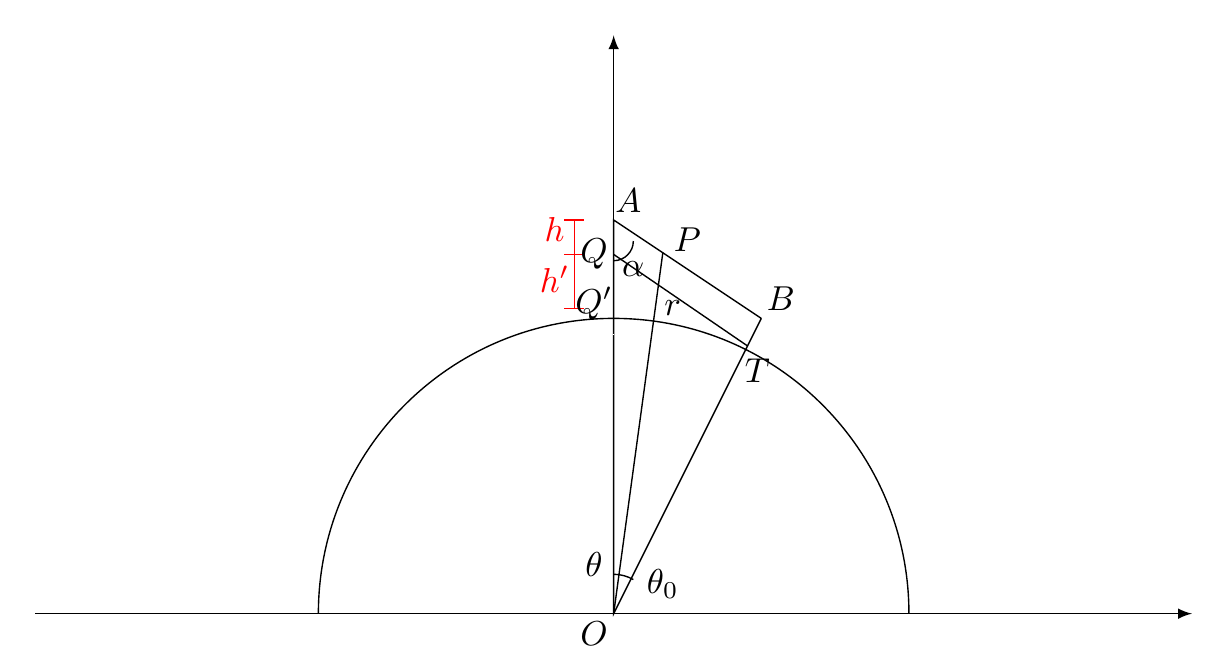}
	\caption{External not intersecting case}
	\label{fig1}
\end{figure}
	We set $z=\pi/2-\alpha$ and we consider the line  $\mathit{r}$ intersecting  the segment $\overline{OA}$,  that is parallel to the line passing through the points $A$ and $B$ and that touches the ball. We call $h$ the length of the segment $\overline{AQ}$, where $Q$ is the point of intersection between the segment $\overline{AO}$ and the line $\mathit{r}$ and we call $h'$ the length of the segment $\overline{QQ'}$, where $Q'$ is the the point of intersection between the segment $\overline{AO}$ and the ball. 	Let us first assume that $\alpha+\theta_0=\pi/2$.  We have that 
	\begin{equation}
	u(\theta)=\dfrac{\cos(\theta_0-\theta)}{\sin(\alpha)(h+h'+1)}.
	\end{equation}
So, setting $I=1/\left(\sin(\alpha)(h+h'+1)\right)$, we obtain that 
\begin{equation}
\int_{0}^{\theta_0} \left( u'(\theta)\right)^2d\theta=I^2\int_{0}^{\theta_0} \sin^2(\theta-\theta_0)\:d\theta=I^2\left( \dfrac{\theta_0}{2}-\dfrac{\sin(2\theta_0)}{4}\right)\approx I^2\frac{4}{3}\theta_0^3.
\end{equation}
	Let us denote by $A_T$ the area of the trapeze $ABTQ$ and let us compute $A_T$:
	\begin{equation}
A_T=\dfrac{\sin(\theta_0)(1+2h+h')}{2} h'\cos(\theta_0)\approx\frac{h'}{2}\sin(2\theta_0)\approx \frac{h'}{2}(2\theta_0-\frac{4\theta^3_0}{3}).
	\end{equation}
	So, since 
	\begin{equation*}
		\ds\int_{0}^{\theta_0}(1-u(\theta)) \;d\theta\geq A_T,
	\end{equation*}
	we can conclude that there exists a costant $C>0$ such that
	\begin{equation*}
		\ds\int_{0}^{\theta_0}|1-u(\theta)| \;d\theta\geq C\int_{0}^{\theta_0} \left( u'(\theta)\right)^2d\theta.
		\end{equation*}\\
		Now let us assume that $\alpha+\theta_0>\pi/2$ and let us set $z=\pi/2-\alpha$.
		 We have that 
		 \begin{equation}
		 	u(\theta)=\dfrac{\sin(\theta+\alpha)}{\sin(\alpha)(h+h'+1)}
		 \end{equation}
		and that 
		\begin{align*}
	&\dfrac{1}{I^2}	\int_{0}^{\theta_0} \left( u'(\theta)\right)^2d\theta=\int_{0}^{\theta_0} \sin^2(\theta+\alpha)\:d\theta=\dfrac{\theta_0}{2}-\dfrac{\sin(2\theta_0+2\alpha)}{4}+\dfrac{\sin(2\alpha)}{4}\\&= \dfrac{\theta_0}{2}+\dfrac{\sin(2z)}{4}\left(  1-\cos(2\theta_0)\right)+\dfrac{cos(2z)\sin(2\theta_0)}{4}\approx\theta_0+\theta_0^2z-z^2\theta_0-\frac{2}{3}\theta_0^3.
			\end{align*}
	If we compute $A_T$:
	\begin{equation*}
A_T=\frac{\sin(\theta_0)}{2}[\sin^2(\alpha)\sin^2(\beta)+\sin^3(\alpha)\sin(\beta)(1+h+h')]
	\end{equation*}
		and we can conclude.\\
		


			\textit{Case $3$: Internal not Touching. }
		Let us assume that the segments $\overline{OA}$ and $\overline{OB}$ are both less than $1$ and we call $h:=1-|\overline{OA}|$ and $\beta$ the angle $O\hat{B}A$
		
	\begin{figure}[htpb]
	\centering
	\includegraphics[scale=0.90]{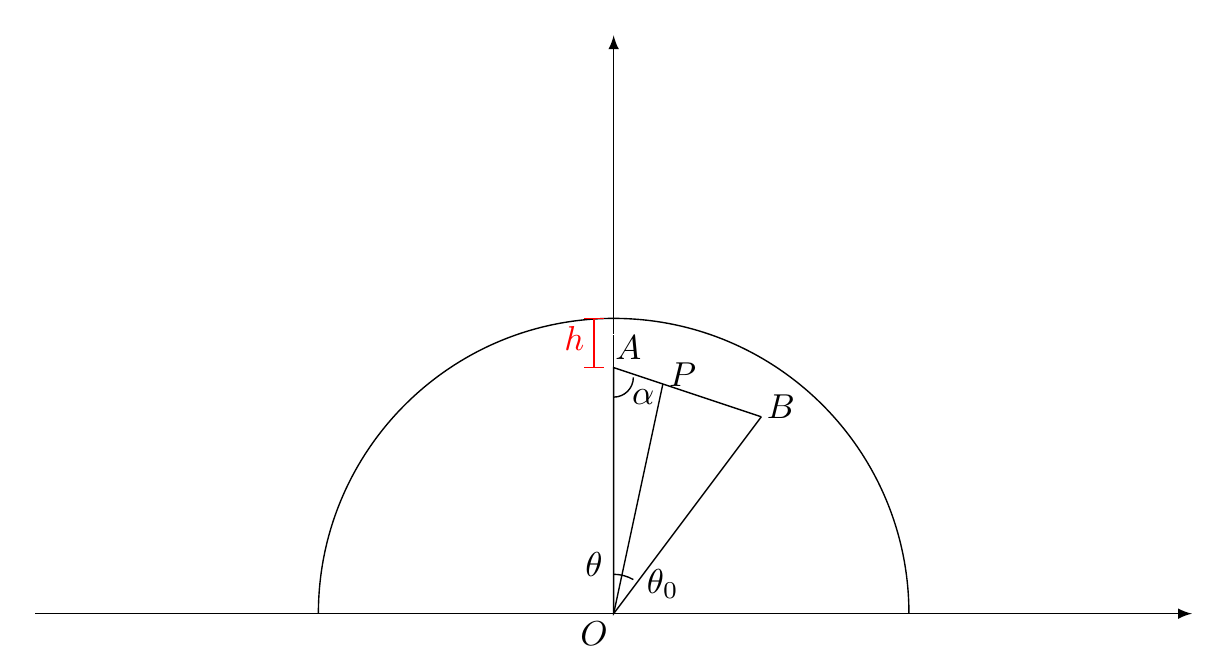}
	\caption{Internal not intersecting  case}
	\label{fig1}
\end{figure}
			
			In this case we have that
			\begin{equation*}
				u(\theta)=\dfrac{\sin(\alpha+\theta)}{(1-h)\sin(\alpha)}.
			\end{equation*}
			Setting $I=1/(1-h)\sin(\alpha)$ and $z=\pi/2-\alpha$,
			\begin{equation*}
			 \dfrac{1}{I^2}	\int_{0}^{\theta_0} \left( u'(\theta)\right)^2d\theta=\dfrac{\theta_0^2}{2}+\dfrac{\sin(2z-2\theta_0)}{4}-\dfrac{\sin(2z)}{4}\approx\dfrac{\theta^3_0}{3}-z\theta_0^2+z^2\theta_0.
			\end{equation*}
			 Let $Q$ be the point that lays in the circumference and  in the semi-line that contains  the segment $\overline{OA}$ and $Q'$ a point that lays in the semi-line containing $\overline{OB}$, such that the segment $\overline{QQ'}$ is  parallel to $\overline{AB}$.
		Calling $A_T$ the area of the trapeze $ABQQ'$, we have
			\begin{equation*}
		A_T=\dfrac{h(2-h)\sin(\alpha)}{2\sin(\beta)}\sin(\theta_0).
			\end{equation*}
			We can conclude, since
		\begin{equation*}
		\ds\int_{0}^{\theta_0}|1-u(\theta)| \;d\theta\geq A_T.
		\end{equation*}

		\textit{Case $4$: Internal Touching. }
		Let us assume that the side $\overline{AB}$ is now contained in the ball and that the segment $\overline{OB}$ has length less than $1$.
			\begin{figure}[htpb]
	\centering
	\includegraphics[scale=0.90]{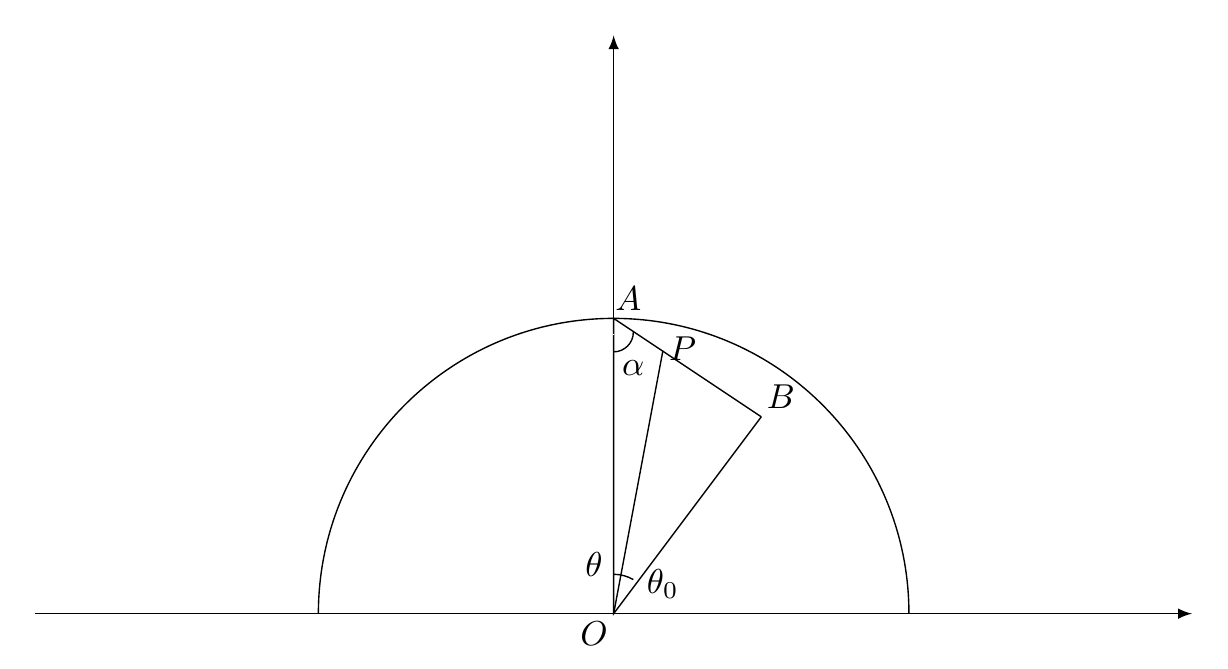}
	\caption{Internal secant case}
	\label{fig1}
\end{figure}
	In this case, setting $I=1/\sin(\alpha)$ and $z=\pi/2-\alpha$, we have, when $z\to 0$ that 
			\begin{align*}
			& \dfrac{1}{I^2}	\int_{0}^{\theta_0} \left( u'(\theta)\right)^2d\theta=\int_{0}^{\theta_0} \cos^2(\theta+\alpha)\:d\theta=\frac{\theta_0}{2}+\dfrac{\sin(2z-2\theta_0)}{4}-\dfrac{\sin(2z)}{4}\\&
			\approx \frac{\theta_0^3}{3}+z^2\theta_0-z\theta_0^2
			\end{align*}
			and 
			\begin{align*}
		& \dfrac{1}{I}\ds\int_{0}^{\theta_0}(u(\theta)-1)\;d\theta=-\sin(z-\theta_0)+\sin(z)-\theta_0\cos(z)\approx\dfrac{z\theta_0^2}{2}-\dfrac{\theta_0^3}{6},
			\end{align*}
			and we can conclude, since we can show after some computations of Euclidean geometry that $z$ goes to $0$ with the same speed as $\theta_0$.
	If we are in the case that $z$ does not tend to $0$, as $\theta_0\to0$, we Taylor expand in $\theta_0$ and we easily  obtain the desired result. 	
\end{proof}

\subsection{Proof of Proposition  \ref{first_step} in the higher dimensional case}
We can  assume that the boundary of  $\Omega\in\mathcal{C}_n$ can be parametrized as  in \eqref{boundary} and we can argue as in Section \ref{planny} (using the results contained in \cite{gelli}): for every $\varepsilon>0$, there exists  a  positive constant $\delta$,  depending only on the dimension $n $ such that, if $\Delta P(\Omega)\leq\delta$ , then $||1-r||_{L^\infty(\mathbb{S}^{n-1})}\leq \varepsilon$. 
Now, from  \eqref{asi_convex},  we have that 
\begin{equation}\label{fuscoo}
\mathcal{A}_F(\Omega)\leq \frac{n+1}{n}\int_{\mathbb{S}^{n-1}}|r(x)-1|\;d\mathcal{H}^{n-1}.
\end{equation}
So, by \eqref{fuscoo} and \eqref{peri_convex},  proving Proposition \ref{first_step}, is equivalent to prove that there exists $C>0$ such that  
 \begin{equation}\label{final_3D}
\int_{\mathbb{S}^{n-1}}|r(x)-1|\;d\mathcal{H}^{n-1}\geq C\int_{\mathbb{S}^{n-1}} \left[   |\nabla_{\tau} r|^2+r^2-\dfrac{1}{r^{2(n-2)}} \right] \;d\mathcal{H}^{n-1}.
 \end{equation}
  We note now that
  \begin{equation*}
     r^2-\dfrac{1}{r^{2(n-2)}}=\dfrac{\left(r-1\right)}{r^{2(n-2)}}  \left(1+r+\cdots r^{2n-3}\right).
      \end{equation*}
     and  the term 
 $$\dfrac{1}{r^{2(n-2)}}\left(1+r+\cdots r^{2n-3}\right) $$ 
is bounded if  $||1-r||_{L^\infty(\mathbb{S}^{n-1})}\leq \varepsilon$,  so that it can absorbed by the constant $C$. For this reason, it remains to prove
\begin{equation*}
    \int_{\mathbb{S}^{n-1}}|r(x)-1|\;d\mathcal{H}^{n-1}\geq C\int_{\mathbb{S}^{n-1}} |\nabla_{\tau} r|^2  \;d\mathcal{H}^{n-1},
\end{equation*}
that is the claim of the following proposition.
 
 \begin{lem}  \label{nd} Let $n>2$. There exist  $C>0$, 
 	depending only on $n$,  such that, for every $\Omega\in\mathcal{C}_n$ with $\Delta P (\Omega)<\delta$, 
 then, if $r$ is the function that describes the boundary of $\Omega$ as defined in \eqref{boundary}, it holds
 \begin{equation}\label{final_ND}
\int_{\mathbb{S}^{n-1}}|r(x)-1|\;d\mathcal{H}^{n-1}(x)\geq C\int_{\mathbb{S}^{n-1}} |\nabla_{\tau} r(x)|^2  \;d\mathcal{H}^{n-1}(x).
 \end{equation}
\end{lem}

\begin{proof}
We fix a system of coordinates $\left(O,e_1,\cdots, e_n\right)$ and we  prove the  desired inequality in a neighbourhood $\omega$ of the point $x_0:=(0,\cdots,0,1)$.\\
Let $x\in\omega$, there exists $c>0$ such that 
\begin{equation}\label{tang}
    |\nabla_\tau r(x)|^2\leq c \sum_{i=1}^{n-1} |\nabla_i r(x)|^2,
\end{equation}
where $\nabla_i u(x)$ is the tangential gradient of $u$ in the circle $C_i$, that is the circle given by the intersection of $\mathbb{S}^{n-1}$ and the plane that contains the origin $O$ and the vectors $x$ and $e_i$, for every $i=1,\dots,n-1$. Inequality \eqref{tang} is an easy consequence of the fact that the tangent $\tau_i$ at $C_i$ in $x$ are almost orthogonal.  
We prove now the following inequality: there exists a constant $C>0$ such that 
\begin{equation}\label{n-1}
    \int_{\omega}|r(x)-1|\;d\mathcal{H}^{n-1}(x)\geq C\int_{\omega} |\nabla_{n-1} r|^2  \;d\mathcal{H}^{n-1}(x).
\end{equation}
In order to do that, we consider the spherical coordinates of $\mathbb{S}^{n-1}$:
\begin{equation*}
\begin{cases}
x_1=\cos(\theta_1)\\
x_2=sin(\theta_1)\cos(\theta_2)\\
\vdots\\
x_{n-1}=\sin(\theta_1)\dots\sin(\theta_{n-2}) \cos(\theta_{n-1})\\
x_n=\sin(\theta_1)\dots\sin(\theta_{n-2})\sin(\theta_{n-1})
\end{cases}
\end{equation*}
with $\theta_1\dots,\theta_{n-2}\in[0,\pi]$ and $\theta_{n-1}\in[0,2\pi)$. 
In this coordinate system the point $x_0=(0,\cdots,0,1)$ corresponds to 
\begin{align*}
    &\theta_1=\pi/2,\\
   & \theta_2=\pi/2,\\
    &\vdots\\
    &\theta_{n-1}=\pi/2.
    \end{align*}
We fix now $\theta_1,\dots, \theta_{n-2}$ and we let $\theta_{n-1}$ vary in a small interval $[\pi/2-\delta_1, \pi/2+\delta_1]$. 
We write on the the piece of circle $\omega'$ described by this parametrization the bidimensional inequality that we have proved in the previous Section \ref{planny}
\begin{equation}\label{ind}
    \int_{\omega'}|r(x)-1|\;d\mathcal{H}^1
\geq C \int_{\omega'}|\nabla_{n-1}r|^2\;d\mathcal{H}^1.
\end{equation}
Now we multiply both sides of \eqref{ind} by the corresponding Jacobian $$\left( \sin^{n-2}(\theta_1)\sin^{n-3}(\theta_2)\dots \sin(\theta_{n-2})\right)$$ and we integrate in $\theta_1,\theta_2,\dots,\theta_{n-2}$ in a small neighbourhood, proving in this way inequality \eqref{n-1}.
In a similar way, with a suitable choice of spherical coordinates, we prove that for  $i=1,\dots, n-2$, also holds
\begin{equation}\label{i}
       \int_{\omega}|r(x)-1|\;d\mathcal{H}^{n-1}(x)\geq C\int_{\omega} |\nabla_{i} r|^2  \;d\mathcal{H}^{n-1}(x).
\end{equation}
Summing and using \eqref{tang},  we can conclude that
 \begin{equation}\label{final_ND}
\int_{\omega}|r(x)-1|\;d\mathcal{H}^{n-1}(x)\geq C\int_{\omega} |\nabla_{\tau} r|^2  \;d\mathcal{H}^{n-1}(x).
 \end{equation}

\end{proof}

\section{Proof of the  Theorem \ref{main_thm}} 

\begin{proof}
	

From \eqref{brasco_thm}, it follows that there exists a constant $K_n>0$ such that 
\begin{equation*}
\Delta \lambda_1(\Omega)\geq K \left(  P(\Omega)  \right)^2-\lambda_1(B).
\end{equation*}
Let us define the function $f_1:\mathbb{R}\to \mathbb{R}$ as
\begin{equation*}
f_1(x)=\dfrac{Kx^2-\lambda_1(B)}{K x^2}.
\end{equation*}
Since $$ \lim_{x\to +\infty}f_1(x)=1 ,$$ 
we have that there exists $M_1>0$  such that, if $x>M_1+P(B)$, then 
 $f_1(x)\geq 1/2$. This is equivalent to say that, if $\Delta P (\Omega)>M_1$, then 
 \begin{equation*}
 	\Delta \lambda_1(\Omega)\geq \dfrac{K}{2} P(\Omega)^2=\dfrac{K}{2}\left( \dfrac{\Delta P (\Omega)+P(B)}{\Delta P(\Omega)}	\right)^2 \Delta P(\Omega)^2.
 \end{equation*} 
If we define now 
\begin{equation*}
f_2(x)=\dfrac{x+P(B)}{x},
\end{equation*}
being $\lim_{x\to +\infty} f_2(x)=1$, we have that there exists $M_2>M_1$ such that, if  $x>M_2$ then $f_2(x)\geq 1/2$. This means that, if  $\Delta P(\Omega)>M_2$ then 
\begin{equation}\label{firstt}
\Delta \lambda_1(\Omega)\geq  \dfrac{K}{8} \Delta P(\Omega)^2.
\end{equation}

Let us assume now that we are in the case $\Delta P(\Omega)< \delta$. We can combine the sharp quantitative Faber-Krahn inequality \eqref{quantitative_v} with Proposition \ref{first_step} and we obtain 
\begin{equation}\label{first_final}
\Delta \lambda_1(\Omega)\geq \bar{C} \mathcal{A}^2_F(\Omega) \geq C^2\;\bar{C} \Delta P(\Omega)^2.
\end{equation}

It remains to study the intermediate case $\delta<\Delta P(\Omega)\leq M_2$. 
Combining \eqref{Melas_Haus} and \eqref{per_Haus}, we obtain
\begin{equation}\label{fourth}
\dfrac{\Delta\lambda_1(\Omega)}{\left(\Delta P(\Omega)\right)^2}\geq \dfrac{M \mathcal{A}_H(\Omega)^{2n^2}}{M^2_2}\geq \dfrac{ M K\left(\Delta P(\Omega) P(\Omega)^{(2-n)/(n-1)} \right)^{2n^2}}{M_2^2}\geq \dfrac{\alpha}{M_2^2},
\end{equation}
where we have assumed, without loss of generality, that the ball centered at the origin is the one that  achieves the minimum in \eqref{hausdorf_def} and 
\begin{equation*}
\alpha:=MK\left(\left(M_2+P(B)\right)^{(2-n)/(n-1)} \delta  \right)^{2n^2}.
\end{equation*}

We conclude the proof of Theorem \ref{main_thm}, choosing the constant $c=c(n)>0$ by taking the minimum between the constant in \eqref{firstt}, \eqref{first_final} and \eqref{fourth}:
\begin{equation}\label{finito}
	c_n= \min\left\{\dfrac{K}{8}, \; C^2\bar{C} ,\;\dfrac{\alpha}{M_2^2}\right\}.
\end{equation}
\end{proof}

We conclude with the following Remark.


\begin{rem} Proceeding in this way, by  using the sharp inequality proved in \cite{BDPV}, we are not able to prove the  conjectured inequality  \eqref{ilias_eq}:
	our leading idea is  indeed to combine th quantitative Faber-Krahn   inequality \eqref{quantitative_v} with an inequality of the form
	\begin{equation}\label{target}
	\mathcal{A}_{\mathcal{F}}(\Omega)\geq C\left( P(\Omega)-P(B)\right)^\delta,
	\end{equation}
	$\delta>0. $  In this case,  the target power in \eqref{target} to prove  \eqref{ilias_eq} would be  $3/4$, but, unfortunately, the best power is $\delta=1$ and the  "bad" sets are  in this case  the polygons, as pointed out in Remark \ref{sharp}.
\end{rem}

\section*{Acknowledgments}
The author would like to thank Professor Dorin Bucur from Universit{\'e}  Savoie Mont Blanc  for proposing the problem and for hosting the author in his department LAMA (Laboratoire de Math{\'e}matiques) for a research period during which they discussed about the problem. Moreover, the author thanks Professor Paolo Salani from University of Florence for the useful and  insightful discussions during the period of his visit at LAMA. 

\small{

}

\end{document}